\documentclass{amsart}
\usepackage{amsfonts}
\usepackage{amsthm}
\usepackage{amsmath, pb-diagram}
\usepackage{amscd}
\usepackage[latin2]{inputenc}
\usepackage{t1enc}
\usepackage[mathscr]{eucal}
\usepackage{indentfirst}
\usepackage{graphicx}
\usepackage{graphics,epsfig}
\usepackage{pict2e}
\usepackage{epic}
\numberwithin{equation}{section}
\usepackage[margin=2.9cm]{geometry}
\usepackage{epstopdf}
\usepackage{xfrac}
\usepackage{float}
\usepackage[shortlabels]{enumitem}
\usepackage{mathtools}
\usepackage{amsmath,gauss,tikz}
\usepackage{lscape}
\usepackage{array}
\usepackage{setspace}
\usepackage{enumitem}

\theoremstyle{plain}
\newtheorem{theo}{Theorem}[section]
\newtheorem{cor}[theo]{Corollary}
\newtheorem{lem}[theo]{Lemma}
\newtheorem{alg}[theo]{Algorithm}
\newtheorem{prop}[theo]{Proposition}

\theoremstyle{definition}
\newtheorem{defn}[theo]{Definition}
\newtheorem{ex}[theo]{Example}

\title{Flow-up Bases for Generalized Spline Modules on Arbitrary Graphs}

\author{Selma Altinok \and Samet Sarioglan}

\address{Selma Altinok, Hacettepe University Department of Mathematics, 06800 Beytepe Ankara Turkey.}
\email{sbhupal@hacettepe.edu.tr}

\address{Samet Sarioglan (Corresponding author), Hacettepe University Department of Mathematics, 06800 Beytepe Ankara Turkey.}
\email{ssarioglan@hacettepe.edu.tr}

\begin{document}
\begin{abstract}
Let $R$ be a commutative ring with identity. An edge labeled graph is a graph with edges labeled by ideals of $R$. A generalized spline over an edge labeled graph is a vertex labeling by elements of $R$, such that the labels of any two adjacent vertices agree modulo the label associated to the edge connecting them. The set of generalized splines forms a subring and module over $R$. Such a module it is called a generalized spline module. We show the existence of a flow-up basis for the generalized spline module on an edge labeled graph over a principal ideal domain by using a new method based on trails of the graph. We also give an algorithm to determine flow-up bases on arbitrary ordered cycles over any principal ideal domain.
\end{abstract}
\maketitle

\section{Introduction}
\label{intro}

Classical splines are piecewise polynomial functions defined on polyhedral complexes that agree up to a smoothness degree at the intersection of faces. They are used in many areas related with industry, computer based animations and geometric design. Classical spline theory is studied by many mathematicians as Alfeld~\cite{Alf}, Schumaker~\cite{Sch}, Billera~\cite{Bil1,Bil2,Bil}, Rose~\cite{Bil,Ros1,Ros2} and Schenck~\cite{Sch}. While classical splines are defined on polyhedral complexes, generalized splines are defined on edge labeled graphs $(G, \alpha)$ with a base ring $R$, where $G$ is a graph and $\alpha$ is an edge labeling function, introduced by Gilbert, Polster and Tymoczko~\cite{Gil}. Billera and Rose~\cite{Bil2} introduced a description of classical splines in terms of dual graph of a polyhedral complex. This is the starting point of generalized spline theory. The set of generalized splines on an edge labeled graph has a ring structure and $R$-module structure as in the case of classical splines. Here we study the module structure of generalized splines. 

The main problem of the theory of generalized splines is to determine whether generalized spline modules are free or not, and if so, to characterize the bases of generalized spline modules. If $R$ is a principal ideal domain, then the generalized spline module $R_{(G,\alpha)}$ over an arbitrary graph $(G,\alpha)$ has a free $R$-module structure. Also if $G$ is a tree, then $R_{(G,\alpha)}$ is free independently of the base ring. 

A special type of splines, which is called flow-up classes, is a useful tool to find module bases for $R_{(G,\alpha)}$. Flow-up classes are first introduced on cycles by Handschy, Melnick and Reinders in 2014,~\cite{Hand}. They studied integer splines and showed the existence of flow-up classes on cycles over $\mathbb Z$. They also proved that the smallest flow-up classes exist and formed a basis for $\mathbb{Z}_{(C_n , \alpha)}$. The smallest leading entries of flow-up classes have a big role to determine a basis. In~\cite{Tym}, Bowden and Tymoczko showed not only the existence of a certain flow-up classes for any graph over the quotient ring $\mathbb{Z} / m \mathbb{Z}$ but also proved that these flow-up classes form minimum $\mathbb Z$-module generators. In~\cite{Phi}, Philbin and the others studied splines on any connected graph and gave an algorithm to find a minimum flow-up generating set for ${\mathbb{Z} / m \mathbb{Z}}_{(C_n,\alpha)}$. They also extended some of their results to $\mathbb{Z}_{(C_n , \alpha)}$.

In this paper we introduce a method to determine the smallest flow-up classes on an arbitrary graph over a principal ideal domain by using special trails, which is a new approach compare to Bowden and Tymoczko in~\cite{Tym}. In order to do this, we compute the smallest leading entries of flow-up classes by combinatorial techniques. The existence of such flow-up classes can be shown by other methods too. The difference and significance of our approach is that: A basis criteria for $R_{(G,\alpha)}$ can be given by a crucial element $Q_G$ of $R$ which is defined by using trails of $G$ (see~\cite{Alt}). It is shown that $Q_G$ corresponds to the product of these smallest leading entries if $G$ is a cycle, a diamond graph or a tree over a principal ideal domain(PID) $R$.  If $R$ is not a PID, then $R_{(G,\alpha)}$ may not have any flow-up basis even if it is free (see an example in Section~\ref{Sflowup}). Nevertheless the element $Q_G$ gives a basis criteria for generalized spline modules.

In the last section we give an algorithm to determine flow-up bases on arbitrary ordered cycles over principal ideal domains as an application of our trail method.

\section{Generalized Splines}
\label{Gsplines}

In this section, we give some basic definitions and properties of generalized splines.

\begin{defn} Let $G = (V,E)$ be a finite graph, $R$ be a commutative ring with identity and $\alpha : E \to \{ \text{ideals in }R \}$ be a function that labels the edges of $G$ with ideals in $R$. We call the pair $(G, \alpha)$ as an edge labeled graph. The ring $R$ is called the base ring.
\end{defn}

Each edge of $(G,\alpha)$ is labeled with a generator of the ideal $I$ when the corresponding ideal $I$ is principal. Throughout the paper we assume that $G$ is a simple and connected graph.

\begin{defn} A generalized spline on an edge labeled graph $(G, \alpha)$ is a vertex labeling $F \in R ^{|V|}$ such that for each edge $v_i v_j \in E$, we have
\begin{displaymath}
f_{i} - f_{j} \in \alpha(v_i v_j)
\end{displaymath}
where $f_{i}$ denotes the label on vertex $v_i$. The collection of all generalized splines on a base ring $R$ over the edge labeled graph $(G, \alpha)$ is denoted by $R_{(G,\alpha)}$.
\end{defn}
From now on we refer to generalized splines as splines.

Let $(G, \alpha)$ be an edge labeled graph with $n$ vertices. We denote the elements of $R_{(G,\alpha)}$ by column matrix notation with ordering from bottom to top as follows:
\begin{displaymath}
F = \begin{bmatrix} f_n \\ \vdots \\ f_1 \end{bmatrix} \in R_{(G,\alpha)}.
\end{displaymath}
We also use vector notation as $F = (f_1 , \ldots , f_n)$.

\begin{ex} Let $(G, \alpha)$ be as the figure below:
\begin{figure}[H]
\centering
\scalebox{0.16}{\includegraphics{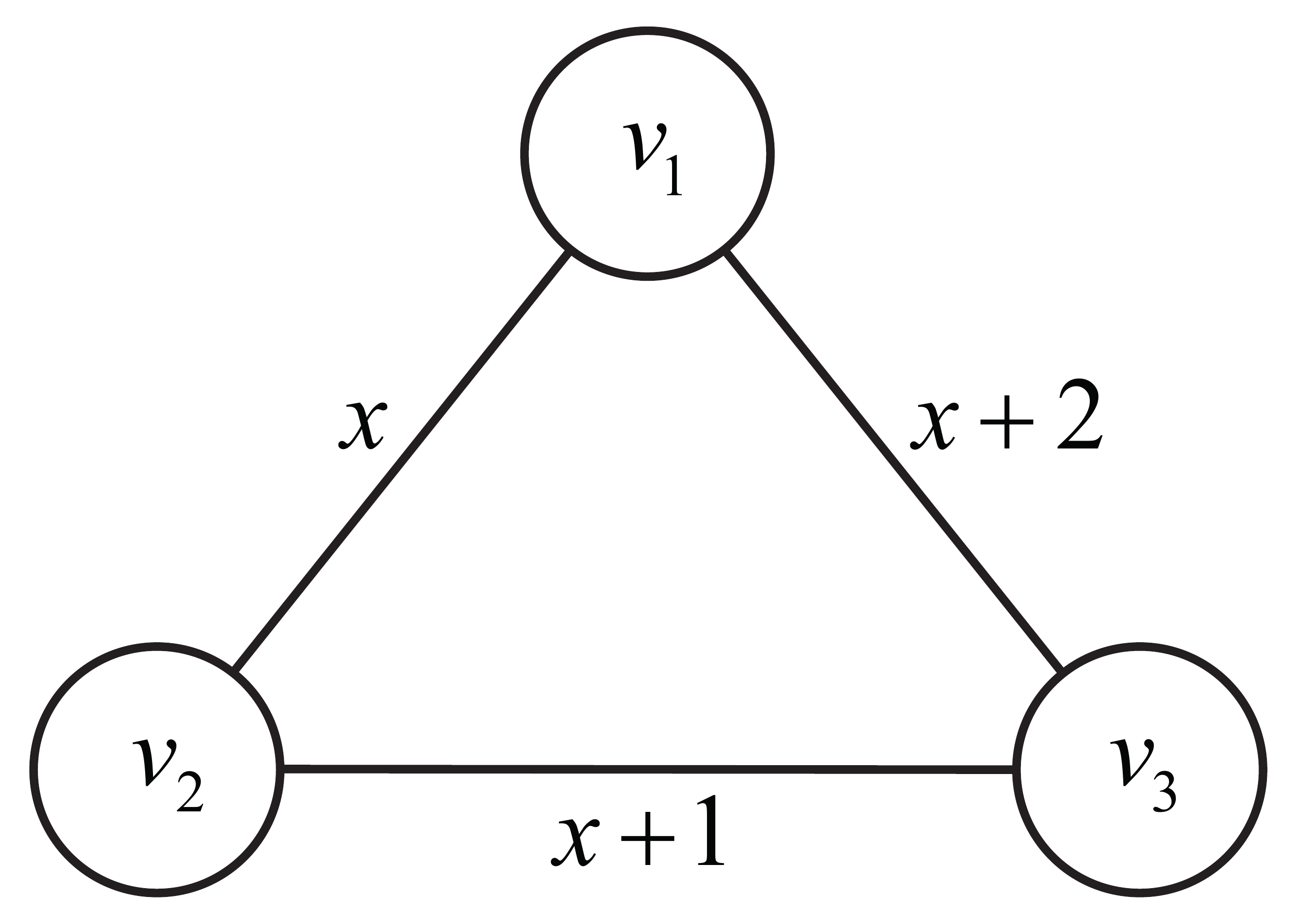}}
\caption{Example of generalized spline}
\label{sp2}
\end{figure}
A spline over $(G, \alpha)$ can be given by
\begin{displaymath}
F = \begin{bmatrix} x^2 + 2x + 1 \\ x+1 \\ 1 \end{bmatrix} .
\end{displaymath}
\end{ex}

If we label all vertices of $(G, \alpha)$ by a fixed $r \in R$, we get a spline since the difference on every edge is zero. We call such splines as trivial splines. Let $e_{ij}$ be the edge that connects the vertices $v_i$ and $v_j$. If $\alpha(e_{ij}) = 1$, then spline condition on $e_{ij}$ holds for any $f_i , f_j \in R$. If $\alpha(e_{ij}) = 0$, then $f_i = f_j$.

If the base ring is an integral domain, then the freeness of $R_{(G,\alpha)}$ can be given by the number of vertices of $G$. To show this, we first need the concept of "rank of a module". The rank of an $R$-module $M$, denoted by $\text{rk } M$, is defined by the maximum number of $R$-linearly independent elements of $M$.

\begin{prop} Let $R$ be an integral domain and $M \subset R^t$ be an $R$-submodule of rank $n$. If $M$ contains a generating set $B$ with $n$ elements, then $M$ is free with basis $B$. 
\label{proprank}
\end{prop}

\begin{proof} See Lemma 8.3 of~\cite{Bil}.
\end{proof}

\begin{prop} Let $R$ be an integral domain. If $M$ is an $R$-module and $N \subset M$ is a submodule, then $\emph{rk } N \leq \emph{rk } M$.
\label{proprank2}
\end{prop}

\begin{proof} Say $\text{rk } N = n$ and let $B = \{ b_1 , \ldots , b_n \} \subset N$ be a maximal linearly independent subset of $N$. Since $B \subset N \subset M$, $B$ is linearly independent in $M$. So $n \leq \text{rk } M$ by the definition of the rank of $M$.
\end{proof}

\begin{theo} If $R$ is an integral domain and $G$ is a graph with $n$ vertices, then $\emph{rk } R_{(G,\alpha)} = n$.
\label{theorank}
\end{theo}

\begin{proof} Since $R_{(G,\alpha)} \subset R^n$ as an $R$-submodule, $\text{rk } R_{(G,\alpha)} \leq n$ by Proposition~\ref{proprank2}. Moreover there exists a set of flow-up classes $A = \{ F^{(1)} , \ldots , F^{(n)} \} \subset R_{(G,\alpha)}$, which is linearly independent since $R$ is an integral domain. Hence $\text{rk } R_{(G,\alpha)} \geq n$ by definition of rank and thus we have $\text{rk } R_{(G,\alpha)} = n$.
\end{proof}

\begin{theo} Let $R$ be an integral domain and $G$ be a graph with $n$ vertices.  The module $R_{(G,\alpha)}$ has a generating set $B$ with $n$ elements if and only if 
it is free.
\label{ranktheo2}
\end{theo}

\begin{proof} If we assume that $R_{(G,\alpha)}$ has a generating set $B$ with $n$ elements,  being free follows directly from Theorem~\ref{theorank} and Proposition~\ref{proprank}.
Now let $R_{(G,\alpha)}$ be free. By Theorem\ref{theorank}, it has  rank $n$. Therefore there exists a generating set (actually a basis) for $R_{(G,\alpha)}$ with $n$ elements.\end{proof}

We introduce a special type of splines, which is called flow-up classes, is a useful tool to find $R$-module bases for $R_{(G,\alpha)}$ below.

\begin{defn} Let $(G, \alpha)$ be an edge labeled graph with $n$ vertices. Fix $i$ with $1 \leq i \leq n$. A flow-up class $F^{(i)}$ is a spline in $R_{(G,\alpha)}$ with first $i-1$ leading zeros, that is, the components $F^{(i)}_i \neq 0$ and $F^{(i)}_j = 0$ for all $j < i$. The set of all $i$-th flow-up classes is denoted by $\mathscr{F}_i$.
\end{defn}

\begin{ex} Let $(G, \alpha)$ be as in the Figure~\ref{sp2}. Example of flow-up classes on $(G, \alpha)$ can be given as
\begin{displaymath}
F^{(1)} = \begin{bmatrix} 1 \\ 1 \\ 1 \end{bmatrix}, F^{(2)} = \begin{bmatrix} -x-2 \\ x \\ 0 \end{bmatrix}, F^{(3)} = \begin{bmatrix} x^2 + 3x +2 \\ 0 \\ 0 \end{bmatrix}.
\end{displaymath}
\end{ex}

It can be easily observed that flow-up classes for all $i$ exist. To see this, let $L$ be the product of all edge labels on $(G, \alpha)$. For all $1 \leq i \leq n$, define a labeling on the vertices of $(G, \alpha)$ with $f_j = 0$ for $j < i$ and $f_t = L$ for $t \geq i$. Hence the resulting labeling gives an element of $\mathscr{F}_i$. This construction of flow-up classes is very trivial and not so useful. We will give another method to construct special types of flow-up classes.

The next theorem shows whenever flow up classes form a basis.

\begin{theo} \emph{\cite{Bow}} Let $R$ be integers and $(G,\alpha)$ be an edge labeled graph with $n$ vertices. The following statements are equivalent:
\begin{enumerate}[label=\emph{(\alph*)}]
	\item The set $\{ F^{(1)} , F^{(2)} , \ldots , F^{(n)}\}$ forms a flow-up basis for $R_{(G,\alpha)}$.
	\item For each flow-up class $G^{(i)} = (0 , \ldots , 0 , g_i , g_{i+1} , \ldots , g_n)$, the entry $g_i$ is a multiple of the entry $F^{(i)} _ {i}$.
\end{enumerate}
\label{basis}
\end{theo}

\begin{proof} See Theorem 3.1 in ~\cite{Bow}.
\end{proof}

Theorem~\ref{basis} shows that the leading entries of flow-up classes has a big role to determine whether a set of flow-up classes forms a basis or not. We define certain trails to determine such leading entries in the next section.

\section{Existence of Special Flow-up Classes}
\label{Sflowup}
 In this section we introduce special trails to determine the smallest leading entries of flow-up classes over an integral domain $R$. We show that we can construct flow-up classes with smallest leading entries when $R$ is a PID.

\subsection{Trails}

\begin{defn} Let $(G,\alpha)$ be an edge labeled graph with $n$ vertices. A trail is a sequence of vertices and edges $v_{i_0} , e_{i_1} , v_{i_1}, \ldots , e_{i_k} , v_{i_k}$ in which no edge is repeated. If an edge $e_{i_j} = v_{i_{j-1}} v_{i_j}$  is labeled by $l_{i_j}$ we shorten a trail notation to $l_{i_1} l_{i_2} \ldots l_{i_k}$. For a fixed vertex $v_i$, a trail $\textbf{p} ^{(i,j)}$ that connects $v_i$ to a vertex $v_j$ is called a $v_j$-trail of $v_i$. If $v_j$ is labeled by zero, then $\textbf{p} ^{(i,j)}$ is called a zero trail of $v_i$. A zero trail of $v_i$ is denoted by $\textbf{p} ^{(i,0)}$ if the vertex index $j$ is not important.
\end{defn}

Let $\textbf{p} ^{(i,j)}$ be a $v_j$-trail of $v_i$. We use the notation $\big( \textbf{p} ^{(i,j)} \big)$ for the greatest common divisor of edge labels on $\textbf{p} ^{(i,j)}$ and $[ \text{ } ]$ for the least common multiple. We denote the set of greatest common divisors of edge labels on $v_j$-trails of $v_i$ by $\big\{\big( \textbf{p} ^{(i,j)} \big) \big\}$.

\begin{ex} Let $(G, \alpha)$ be the edge labeled graph in Figure~\ref{sp6} and $(0,0,f_3,f_4,f_5) \in \mathscr{F}_3$. The red and blue arrays illustrates the zero trails of $v_3$. 

\begin{figure}[H]
\begin{center}
\scalebox{0.16}{\includegraphics{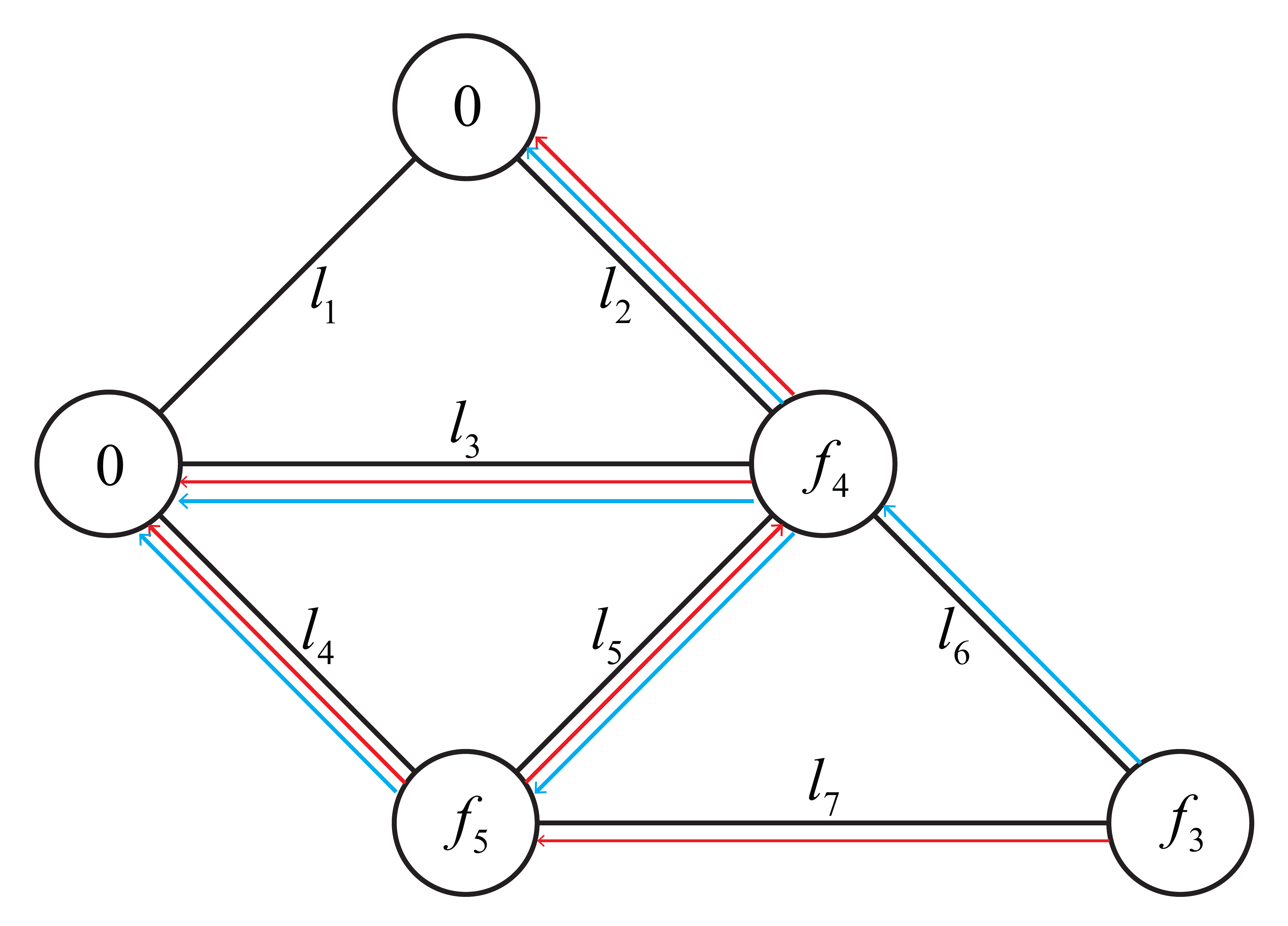}}
\caption{Zero trails}
\label{sp6}
\end{center}
\end{figure}

The zero trails of $v_3$ are listed below:

\begin{displaymath}
\begin{gathered}
\textbf{p} ^{(3,0)}_1 = l_7 l_4, \quad 
\textbf{p} ^{(3,0)}_2 = l_7 l_5 l_3, \quad
\textbf{p} ^{(3,0)}_3 = l_7 l_5 l_2, \quad
\textbf{p} ^{(3,0)}_4 = l_6 l_3, \quad
\textbf{p} ^{(3,0)}_5 = l_6 l_2, \quad
\textbf{p} ^{(3,0)}_6 = l_6 l_5 l_4.
\end{gathered}
\end{displaymath}

\noindent
By notation above, 
\begin{align*}
\big\{\big( \textbf{p} ^{(3,0)} \big) \big\} &= \big\{ \big(\textbf{p} ^{(3,0)}_1\big),\big(\textbf{p} ^{(3,0)}_2),\big(\textbf{p} ^{(3,0)}_3\big),\big(\textbf{p} ^{(3,0)}_4\big),\big(\textbf{p} ^{(3,0)}_5\big),\textbf{p} ^{(3,0)}_6\big)\big\}\\ &=\big\{ (l_7,l_4), (l_7,l_5,l_3), (l_7,l_5,l_2), (l_6,l_3), (l_6,l_2), (l_6,l_5,l_4) \big\}.
\end{align*}

Consider the spline conditions induced by zero trails. For instance, for the zero trail $l_7 l_4$, we have the following conditions.
\begin{displaymath}
\begin{aligned}
f_3 &\equiv f_5 \text{ mod } l_7 \\
f_5 &\equiv 0 \text{ mod } l_4.
\end{aligned}
\end{displaymath}
It implies that
\begin{displaymath}
\begin{aligned}
f_5 &= k_4 l_4 \\
f_3 &= f_5 + k_7 l_7 = k_4 l_4 + k_7 l_7
\end{aligned}
\end{displaymath}
for some $k_4 , k_7 \in R$. Hence $(l_7,l_4)$ divides $k_4 l_4 + k_7 l_7 = f_3$. This holds also for other zero trails of $v_3$.
\end{ex}

This observation leads us to the following proposition.

\begin{prop} Let $(G,\alpha)$ be an edge labeled graph with $n$ vertices and let $F^{(i)} = (0, \ldots ,0, f_i ,\ldots , f_n)\linebreak \in \mathscr{F}_i$  with $i > 1$. Let $v_j$ be a vertex with $j \geq i$ and let $\emph{\textbf{p}} ^{(j,0)}$ be an arbitrary zero trail of $v_j$. Then $\big( \emph{\textbf{p}} ^{(j,0)} \big)$ divides $f_j$.
\label{prop41}
\end{prop}

\begin{cor} Let $F^{(i)} = (0, \ldots ,0, f_i , \ldots , f_n) \in \mathscr{F}_i$ with $i > 1$ on an edge labeled graph with $n$ vertices. Let $v_j$ be a vertex with $j \geq i$ and let $\emph{\textbf{p}} ^{(j,0)}_1, \ldots , \emph{\textbf{p}} ^{(j,0)}_t$ denote zero trails of $v_j$. Then
\begin{displaymath}
\big[ \big\{ \big( \emph{\textbf{p}} ^{(j,0)}_k \big) | 1 \leq k \leq t \big\} \big]
\end{displaymath}
divides $f_j$.
\label{cor48}
\end{cor}

In general, we may not find a flow-up class $F^{(i)} \in \mathscr{F}_i$ with leading element $f_i = \big[ \big\{  \big( \textbf{p} ^{(i,0)} \big) \big\} \big]$. If such a flow-up class exists, then $f_i$ is called the smallest leading entry of the elements of $\mathscr{F}_i$. If $i = 1$, then we can set the smallest leading entry of the elements of $\mathscr{F}_1$ as $1$. Notice that the smallest leading entry satisfies the condition in Theorem~\ref{basis} (b) for all $i$. As we will prove at the end of this section; we can always construct flow-up classes $F^{(i)} \in \mathscr{F}_i$ with the smallest leading entry if the base ring $R$ is a PID. So in this case, zero trails are sufficient to determine the leading entries. For $j > i$ they are not enough to determine the other entries $f_j$ of a flow-up class $F^{(i)}$. In order to illustrate this fact, consider the following example.

\begin{ex} Let $(G,\alpha)$ be an edge labeled graph as in Figure~\ref{sp64} and let $F^{(3)}= (0,0,g_3 , g_4 , g_5 , g_6, g_7)$ be a flow-up class with the smallest leading entry. We try to determine the entries of $F^{(3)}$.

\begin{figure}[H]
\begin{center}
\scalebox{0.16}{\includegraphics{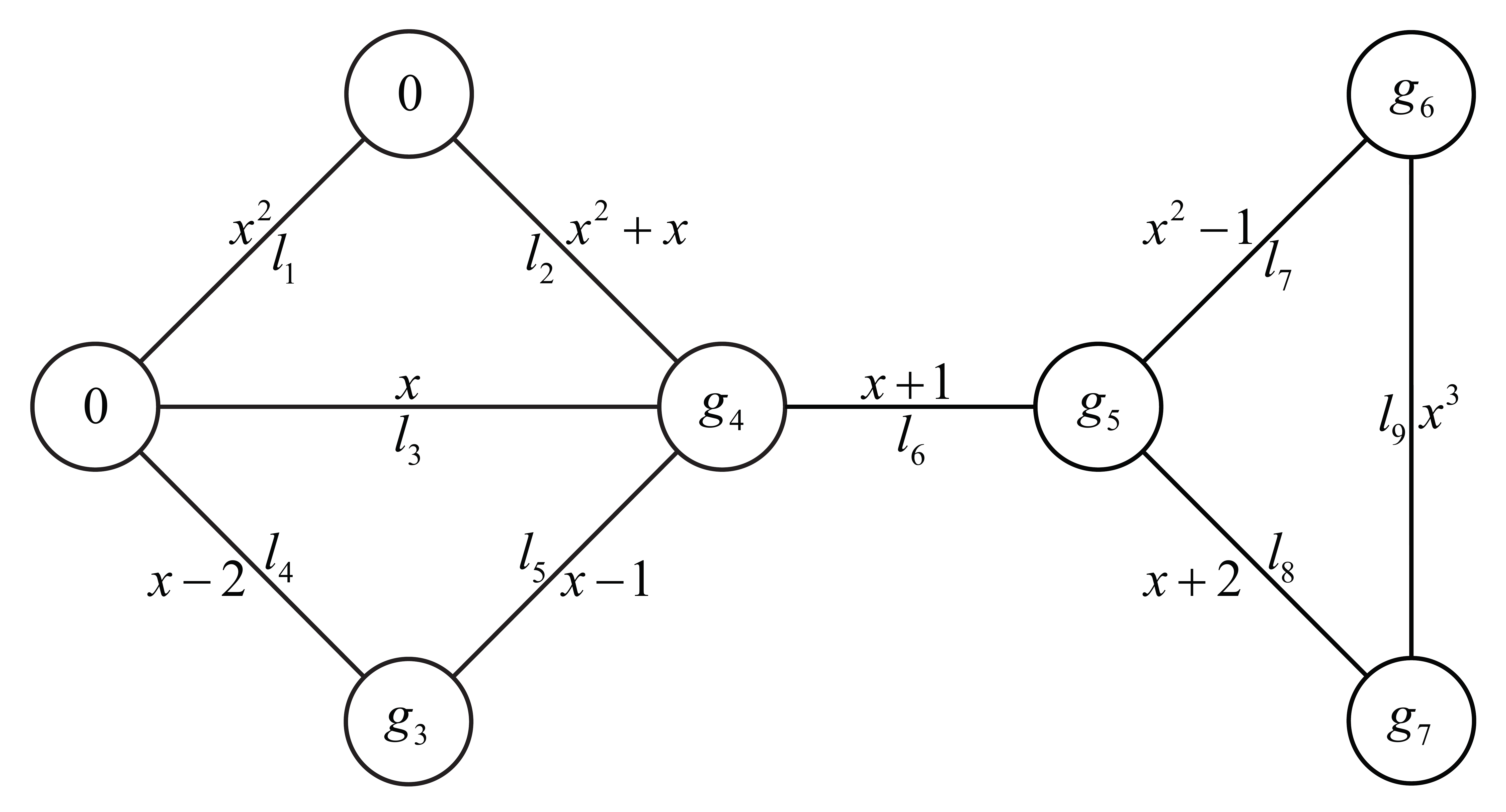}}
\caption{Examle~\ref{exxx}}
\label{sp64}
\end{center}
\end{figure}

\noindent
Zero trails of $v_3$ are $l_4$, $l_5 l_3$ and $l_5 l_2$. So by Proposition~\ref{prop41},
\begin{displaymath}
\begin{aligned}
x - 2 &\text{ $\big\vert$ } g_3 \\
\big( x-1 , x \big) = 1 &\text{ $\big\vert$ } g_3 \\
\big( x-1 , x^2 + x \big) = 1 &\text{ $\big\vert$ } g_3
\end{aligned}
\end{displaymath}
and the smallest value for $g_3 = \big[ x-2 , 1 , 1 \big] = x - 2$.

Now we try to determine $g_4$ by zero trails. Zero trails of $g_4$ are $l_2$, $l_3$ and $l_5 l_4$. So
\begin{displaymath}
\begin{aligned}
x^2 + x &\text{ $\big\vert$ } g_4 \\
x &\text{ $\big\vert$ } g_4 \\
\big( x-1 , x - 2 \big) = 1 &\text{ $\big\vert$ } g_4
\end{aligned}
\end{displaymath}
and then $\big[ x^2 + x , x \big] = x^2 + x$ divides $g_4$. But if we choose $g_4 = x^2 + x$, then this $g_4$ does not satisfy spline conditions since $x^2 + x \not\equiv x - 2 \text{ mod } x - 1$, namely $g_4 \not\equiv g_3 \text{ mod } l_5$. Here we also need to consider $v_3$-trails together with zero trails of $v_4$ to determine $g_4$. The $v_3$-trails of $v_4$ are $l_5$, $l_3 l_4$ and $l_2 l_1 l_4$. So 
\begin{displaymath}
\begin{aligned}
g_4 &\equiv g_3 \text{ mod } l_5 \\
g_4 &\equiv g_3 \text{ mod } \big( l_3 , l_4 \big) \\
g_4 &\equiv g_3 \text{ mod } \big( l_2 , l_1 , l_4 \big) .
\end{aligned}
\end{displaymath}
Hence $g_4 \equiv x - 2 \text{ mod } x - 1$. Together with the factor of $g_4$ given by zero trails of $v_4$, we have
\begin{displaymath}
\begin{aligned}
g_4 &\equiv 0 \text{ mod } x^2 + x \\
g_4 &\equiv x - 2 \text{ mod } x - 1 .
\end{aligned}
\end{displaymath}
Here $g_4 = -(x^2 + x)/2$ satisfies these conditions.

It is sufficient to check just $v_4$-trails of $v_5$ to assign $g_5$ by Lemma~\ref{sufflem2} since all zero trails and $v_3$-trails of $v_5$ passes through $v_4$. Moreover we can consider only $l_6$ instead of all $v_4$-trails of $v_5$ by Lemma~\ref{sufflem}. Hence
\begin{displaymath}
g_5 \equiv g_4 \text{ mod } x+1
\end{displaymath}
and we can assign $g_5 = x+1$.

In order to find $g_6$, we only need $v_5$-trails of $v_6$ by Lemma~\ref{sufflem2} again. Thus
\begin{displaymath}
g_6 \equiv g_5 \text{ mod } \big[ x^2 - 1 , (x^3 , x+2)  \big]
\end{displaymath}
and we can assign $g_6 = 0$.

Finally we have to consider $v_5$-trails and $v_6$-trails of $v_7$ to determine $g_7$.
\begin{displaymath}
\begin{aligned}
g_7 &\equiv g_5 \text{ mod } x + 2 \\
g_7 &\equiv g_6 \text{ mod } x^3 .
\end{aligned}
\end{displaymath}
and we can assign $g_7 = x^3 / 8$.
\label{exxx}
\end{ex}

While determining entries of a flow-up class $F^{(i)}$ with $i \geq 2$ we first set $f_i = \big[ \big\{  \big( \textbf{p} ^{(i,0)} \big) \big\} \big]$. Then we determine $f_{i+1} , \ldots , f_n$ inductively by considering trails to vertices with smaller indices. The following two lemmas show that we do not have to consider all $v_j$-trails of a vertex $v_k$ to determine $f_k$ with $i \leq k$.

\begin{lem} Let $(G,\alpha)$ be an edge labeled graph and let $v_i , v_j \in V(G)$ with $j < i$. Let $\emph{\textbf{p}} ^{(i,j)} _1$ and $\emph{\textbf{p}} ^{(i,j)} _2$ be two $v_j$-trails of $v_i$ with $\emph{\textbf{p}} ^{(i,j)} _1 \subset \emph{\textbf{p}} ^{(i,j)} _2$. Then we do not need to consider $\emph{\textbf{p}} ^{(i,j)} _2$ to determine $f_i$.
\label{sufflem}
\end{lem}

\begin{proof} It is clear that $(\textbf{p} ^{(i,j)} _2)$ divides $(\textbf{p} ^{(i,j)} _1)$ since $\textbf{p} ^{(i,j)} _1 \subset \textbf{p} ^{(i,j)} _2$. Hence the spline condition from $\textbf{p} ^{(i,j)} _2$ is already satisfied by the spline condition from $\textbf{p} ^{(i,j)} _1$ and we do not need to consider $\textbf{p} ^{(i,j)} _2$.
\end{proof}

In consideration of Lemma~\ref{sufflem}, we can refer zero trails as zero paths.

\begin{lem} Let $(G,\alpha)$ be an edge labeled graph and let $\emph{\textbf{p}} ^{(i,j)}$ be a $v_j$-trail of $v_i$ with $j < i$. If there exists a vertex $v_k$ on $\emph{\textbf{p}} ^{(i,j)}$ with $k < i$ then it is sufficient to consider the trail $\emph{\textbf{p}} ^{(i,k)} \subset \emph{\textbf{p}} ^{(i,j)}$ instead of $\emph{\textbf{p}} ^{(i,j)}$ to determine $f_i$.
\label{sufflem2}
\end{lem}

\begin{proof} We show that $f_i \equiv f_k \text{ mod } (\textbf{p} ^{(i,k)})$ implies $f_i \equiv f_j \text{ mod } (\textbf{p} ^{(i,j)})$. If $j > k$, then we already have $f_k \equiv f_j \text{ mod } (\textbf{p} ^{(k,j)})$ by the construction of $f_k$. Therefore $f_i - f_k = r_1 (\textbf{p} ^{(i,k)})$ and $f_k - f_j = r_2 (\textbf{p} ^{(k,j)})$ for some $r_1 , r_2 \in R$. Hence $f_i - f_j = r_1 (\textbf{p} ^{(i,k)}) - r_2 (\textbf{p} ^{(j,k)})$. Notice that $\textbf{p} ^{(i,k)} \cup \textbf{p} ^{(k,j)} = \textbf{p} ^{(i,j)}$, so $(\textbf{p} ^{(i,j)})$ divides $(\textbf{p} ^{(i,k)})$ and $(\textbf{p} ^{(k,j)})$. Thus we have $(\textbf{p} ^{(i,j)})$ divides $f_i - f_j$ and $f_i \equiv f_j \text{ mod } (\textbf{p} ^{(i,j)})$. If $k < j$, then $f_j \equiv f_k \text{ mod } (\textbf{p} ^{(j,k)})$ by the construction of $f_j$ and the proof follows similarly.
\end{proof}

As the main theorem of this section, we prove the existence of flow-up bases for $R_{(G,\alpha)}$ when the base ring $R$ is a PID.

\begin{theo} Let $(G,\alpha)$ has $n$ vertices and $R$ be a PID. Fix $v_i$ with $i > 1$ and assume that all vertices $v_j$ with $j < i$ are labeled by zero. Then a flow-up class $F^{(i)}$ exists with the first nonzero entry $f_i = \big[ \big\{ \big( \emph{\textbf{p}} ^{(i,0)} \big) \big\} \big]$.
\label{thm410}
\end{theo}

\begin{proof}
Let $f_i = \big[ \big\{ \big( \emph{\textbf{p}} ^{(i,0)} \big) \big\} \big]$. We claim the existence of the entries $f_{i+1}, \ldots , f_n$. Assume that $G$ is connected. Otherwise we think of each connected component of $G$ seperately.

We use induction. The following modular equations yields the existence of $f_{i+1}$:
\begin{align*}
f_{i+1} &\equiv 0 \text{ mod } \big[ \big\{ (\textbf{p} ^{(i+1,0)}) \big\} \big] \\
f_{i+1} &\equiv f_i \text{ mod } \big[ \big\{ (\textbf{p} ^{(i+1,i)}) \big\} \big]
\end{align*}
Say $\big[ \big\{ (\textbf{p} ^{(i+1,0)}) \big\} \big] = *$ and $\big[ \big\{ (\textbf{p} ^{(i+1,i)}) \big\} \big] = **$. There exists $f_{i+1}$ satisfying these equations if $(* , **)$ divides $f_i$ by Chinese Remainder Theorem. To see that $(* , **)$ divides $f_i$, let $(* , **) = {p_1}^{\alpha_1} {p_2}^{\alpha_2} \ldots {p_m}^{\alpha_m}$. Choose ${p_j}^{\alpha_j}$ with $1 \leq j \leq m$. Then
\begin{displaymath}
{p_j}^{\alpha_j} \bigm\vert (* , **) \\ \Rightarrow \exists \alpha_{j_1} , \alpha_{j_2} \in R  \text{ such that } \hspace{.19cm} {p_j}^{\alpha_{j_1}} \bigm\vert *, \hspace{.19cm} {p_j}^{\alpha_{j_2}} \bigm\vert ** , \alpha_j = \text{min}(\alpha_{j_1} , \alpha_{j_2})
\end{displaymath}
and hence
\begin{align*}
{p_j}^{\alpha_{j_1}} \bigm\vert * \hspace{.19cm} &\Rightarrow \hspace{.19cm} \exists \textbf{p} ^{(i+1,0)} \hspace{.19cm} \text{ with } \hspace{.19cm} {p_j}^{\alpha_{j_1}} \bigm\vert (\textbf{p} ^{(i+1,0)}) \\
{p_j}^{\alpha_{j_2}} \bigm\vert ** \hspace{.19cm} &\Rightarrow \hspace{.19cm} \exists \textbf{p} ^{(i+1,i)} \hspace{.19cm} \text{ with } \hspace{.19cm} {p_j}^{\alpha_{j_2}} \bigm\vert (\textbf{p} ^{(i+1,i)}) .
\end{align*}
Since $\alpha_j = \text{min}(\alpha_{j_1} , \alpha_{j_2})$, we get
\begin{align*}
{p_j}^{\alpha_j} \bigm\vert (\textbf{p} ^{(i+1,0)}) \hspace{.19cm} \wedge \hspace{.19cm} {p_j}^{\alpha_j} \bigm\vert (\textbf{p} ^{(i+1,i)}) \hspace{.19cm} &\Rightarrow \hspace{.19cm} {p_j}^{\alpha_j} \bigm\vert \big( (\textbf{p} ^{(i+1,0)}) , (\textbf{p} ^{(i+1,i)}) \big)\vspace{.2cm} \\ &\Rightarrow \hspace{.19cm} {p_j}^{\alpha_j} \bigm\vert \big(\textbf{p} ^{(i+1,0)} \cup \textbf{p} ^{(i+1,i)} \big) .
\end{align*}
One can go from $v_i$ to $v_{i+1}$ by $\textbf{p} ^{(i+1,i)}$ and then goes from $v_{i+1}$ to $0$ by $\textbf{p} ^{(i+1,0)}$, so the union $\textbf{p} ^{(i+1,0)} \cup \textbf{p} ^{(i+1,i)}$ is either a zero trail of $v_i$ or contains a zero trail of $v_i$, say $\textbf{p} ^{(i,0)} \subseteq \textbf{p} ^{(i+1,0)} \cup \textbf{p} ^{(i+1,i)}$. The union may have some cycles, so it is not a zero trail of $v_i$, but contains a zero trail of $v_i$ in this case. Here the trail $\textbf{p} ^{(i+1,0)} \cup \textbf{p} ^{(i+1,i)}$ has greater or equal number of edges than the trail $\textbf{p} ^{(i,0)})$ and thus
\begin{displaymath}
{p_j}^{\alpha_j} \bigm\vert \big(\textbf{p} ^{(i+1,0)} \cup \textbf{p} ^{(i+1,i)} \big) \Bigm\vert (\textbf{p} ^{(i,0)}) \bigm\vert f_i .
\end{displaymath}
Since ${p_j}^{\alpha_j}$ is chosen arbitrary, we have ${p_j}^{\alpha_j} \bigm\vert f_i$ for all $1 \leq j \leq m$ and hence
\begin{displaymath}
( * , ** ) = \big[ {p_1}^{\alpha_1} , \ldots , {p_m}^{\alpha_m} \big] \bigm\vert f_i.
\end{displaymath}
This proves the existence of the entry $f_{i+1}$ and $f_{i+1}$ is unique up to modulo $[*,**]$. In other words, $f_{i+1}$ can be chosen as smallest relative to $f_i$.

As an induction hypothesis, suppose that there exist vertex labels $f_{i+2}, f_{i+3}, \ldots , f_{n-1}$ satisfying the spline conditions. The existence of the vertex label $f_n$ is related to the following modular equations:
\begin{align*}
f_{n} &\equiv 0 \text{ mod } \big[ \big\{  (\textbf{p} ^{(n,0)}) \big\} \big] \\
f_{n} &\equiv f_i \text{ mod } \big[ \big\{  (\textbf{p} ^{(n,i)}) \big\} \big] \\
f_{n} &\equiv f_{i+1} \text{ mod } \big[ \big\{ (\textbf{p} ^{(n,i+1)}) \big\} \big] \\
&\vdotswithin{=} \\
f_{n} &\equiv f_{n-1} \text{ mod } \big[ \big\{  (\textbf{p} ^{(n,n-1)}) \big\} \big]
\end{align*}
Let $L_t = \Big[ \Big\{   \big(\textbf{p} ^{(n,t)} \big) \Big\} \Big]$ for $t = 0,i,i+1,\ldots,n-1$. There exists $f_n$ satisfying these equations if
\begin{align}
f_j &\equiv 0 \text{ mod } \big( L_j , L_0 \big) \text{ for all } j \in \{i , i+1 , \ldots , n-1 \} \label{eq1}\\
f_j &\equiv f_k \text{ mod } \big( L_j , L_k \big) \text{ for all } j,k \in \{i , i+1 , \ldots , n-1 \} \label{eq2}
\end{align}
by Chinese Remainder Theorem. 

To conclude that Equation (3.1) holds, first take a factor $p^{\alpha}$ of $\big( L_j , L_0 \big)$ and see that  $p^{\alpha}$ divides $f_j$. Here notice that if $\textbf{p} ^{(n,j)}$ is a $v_j$ trail of $v_n$ and $\textbf{p} ^{(n,0)}$ is a zero trail of $v_n$, then $\textbf{p} ^{(n,j)} \cup \textbf{p} ^{(n,0)}$ is either a zero trail of $v_j$ or contains a zero trail of $v_j$. Hence one can conclude that Equation (3.2) holds by the same observation as in the proof of the existence of $f_{i+1}$. Similarly, notice that if $\textbf{p} ^{(n,j)}$ and $\textbf{p} ^{(n,k)}$ are $v_j$ and $v_k$-trails of $v_n$ respectively with $j,k \in \{i , i+1 , \ldots , n-1 \}$, then $\textbf{p} ^{(n,j)} \cup \textbf{p} ^{(n,k)}$ is either a $v_j$-trail of $v_k$ or contains a $v_j$-trail of $v_k$. So it can be shown that Equation (3.2) holds by taking a factor of $\big( L_j , L_k \big)$ and the same observation. Thus we conclude the existence of $f_n$.

As the last part of the proof, we need show that this construction gives a spline. Take an adjacent pair $v_i , v_j$ of vertices of $(G,\alpha)$ with $i < j$. The edge $e_{ij}$ is a $v_i$-trail of $v_j$. Then we have
\begin{displaymath}
f_j \equiv f_i \text{ mod } \big[ \big\{  (\textbf{p} ^{(j,i)}) \big\} \big]
\end{displaymath}
by construction and hence
\begin{displaymath}
l_{ij} \text{ $\vert$ } \big[ \big\{  (\textbf{p} ^{(j,i)}) \big\} \big] \text{ $\vert$ } f_j - f_i.
\end{displaymath}
\end{proof}

\begin{cor} Let $(G,\alpha)$ be an edge labeled graph with $n$ vertices. If the base ring $R$ is a PID, then there exists a flow-up basis $\{ F^{(1)} , \ldots , F^{(n)} \}$ where $F^{(i)} _i = \big[ \big\{ \big( \textbf{p} ^{(i,0)} \big) \big\} \big]$ for $1 < i \leq n$ and $F^{(1)} = (1 , \ldots , 1)$.
\end{cor}

\begin{proof} It follows directly from Theorem~\ref{basis}.
\end{proof}

It is acceptable that using zero trails method for huge graphs is not practical, however a basis criteria for spline modules over some specific graph types is given in~\cite{Alt} by using this method. The element $f_i \in R$ defined in Theorem~\ref{thm410} has a crucial role to give the criteria. In~\cite{Alt}, we define the element $Q_G = \prod_{i = 1} ^n f_i \in R$ where $n = |V(G)|$ to determine whether a given set of splines forms a basis for $R_{(G, \alpha)}$ or not, where $G$ is a cycle, a diamond graph or a tree. As an open question, we claim that a basis criteria for any graph $G$ over a GCD domain can be given by the element $Q_G$.

We conclude that if $R$ is PID, then we can construct flow-up basis for $R_{(G, \alpha)}$. But if $R$ is not a PID, then $R_{(G, \alpha)}$ may not have a flow-up basis even it is free. The following example illustrates this case.

\begin{ex} Let $(G,\alpha)$ be as the figure below and let $R = k[x,y]$.

\begin{figure}[H]
\begin{center}
\scalebox{0.16}{\includegraphics{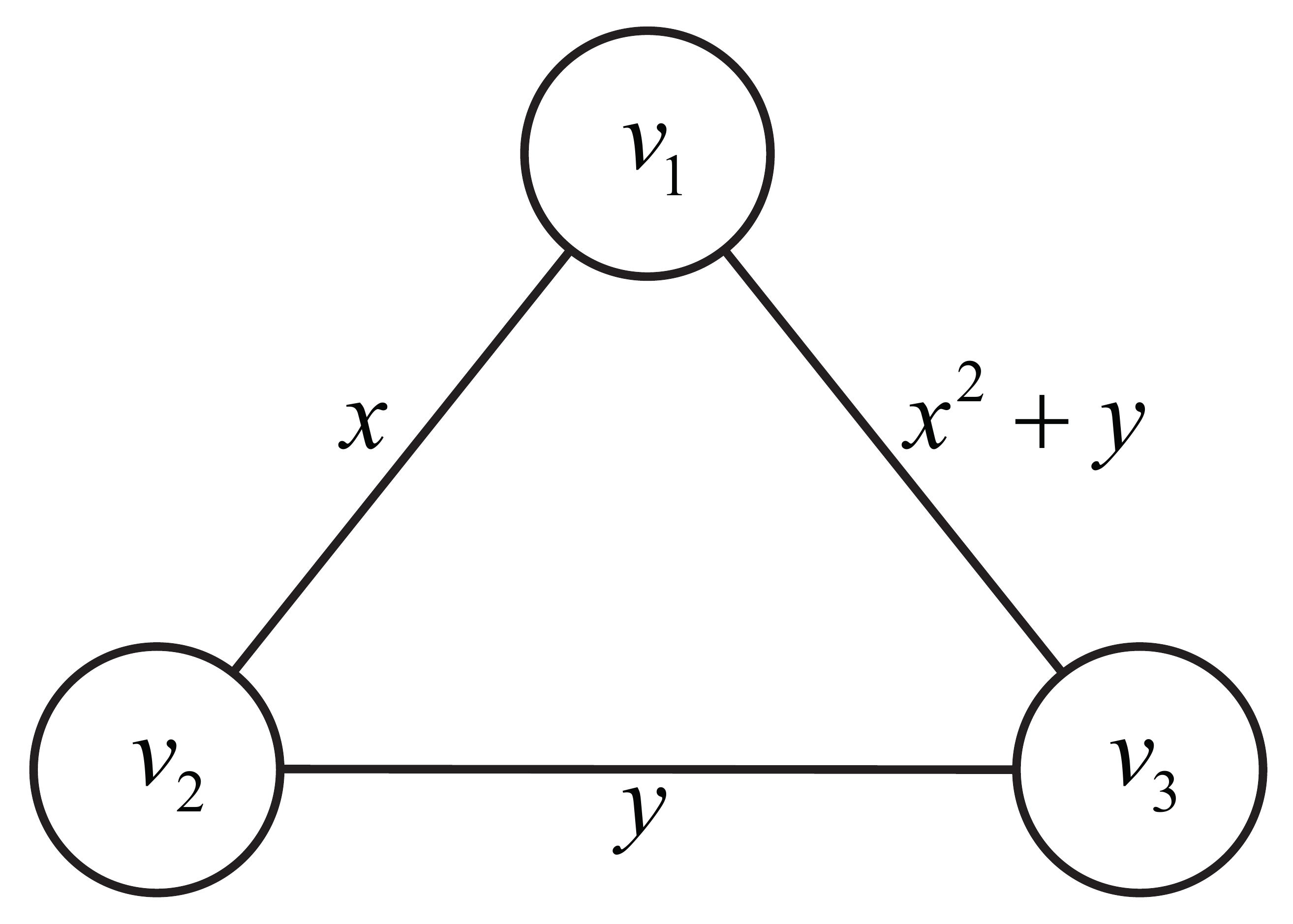}}
\caption{Edge labeled $3$-cycle}
\label{sp2112}
\end{center}
\end{figure}

\noindent
If we consider the flow-up class $F^{(2)} = (0,f_2,f_3)$ with smallest leading entry, we can compute\linebreak $\big[ \big\{ \big( \textbf{p} ^{(2,0)} \big) \big\} \big] = x$ but there is no $f_3 \in R$ satisfying the spline conditions for $f_2 = x$. It means that there is no flow-up basis for $R_{(G,\alpha)}$. For example,  choose a set $\mathcal{A}$ of flow-up classes
\begin{displaymath}
\mathcal{A} = \left\{ \begin{bmatrix} 1 \\ 1 \\ 1 \end{bmatrix} , \begin{bmatrix} 0 \\ xy \\ 0 \end{bmatrix} , \begin{bmatrix} x^2 y + y^2 \\ 0 \\ 0 \end{bmatrix} \right\} \subset R_{(G,\alpha)}.
\end{displaymath}
The set $\mathcal{A}$ is not a basis for $R_{(G,\alpha)}$. In order to see this, notice that $F = (0,x^2,x^2 + y) \in R_{(G,\alpha)}$ is a flow-up class and the first nonzero entry of $F$ is not a multiple of the first nonzero entry of $(0, xy , 0) \in \mathcal{A}$. Hence $\mathcal{A}$ is not a basis for $R_{(G,\alpha)}$ by Theorem~\ref{basis}. It is not so difficult to see that  the same result holds for other choice of  $\mathcal{A}$ for this example. 

We can compute $R$-module generators of $R_{(G,\alpha)}$ by Macaulay2 as
\begin{displaymath}
\mathcal{G} = \left\{ \begin{bmatrix} 1 \\ 1 \\ 1 \end{bmatrix} , \begin{bmatrix} 0 \\ y \\ x^2 + y \end{bmatrix} , \begin{bmatrix} 0 \\ xy \\ 0 \end{bmatrix} \right\}.
\end{displaymath}
Since $|\mathcal{G}| = |V(G)|$, we can conclude that $R_{(G,\alpha)}$ is a free $R$-module by Theorem~\ref{ranktheo2}. By using determinantal techniques with $Q_G=1\cdot x\cdot y(x^2+y) \in R$ discussed in~\cite{Alt} it can not only be shown that $\mathcal{G}$ is a $R$-module basis for $R_{(G,\alpha)}$  but also any flow-up classes are not a basis.
\end{ex}
Spline modules on trees always have a flow-up basis even $R$ is not a PID. For a detailed proof of this fact, see Theorem 4.1 in~\cite{Gil}.

In the next section we give an algorithm to determine flow-up classes on arbitrary ordered cycles.

\section{Flow-Up Classes on Arbitrary Ordered Cycles}
\label{arbordcyc}

Bowden, Hagen, King and Reinders studied flow-up classes on ordered cycles over $\mathbb{Z}$ in~\cite{Bow}. They gave explicit formulas for the entries of flow-up classes with smallest leading entry on ordered cycles. In this chapter we first give the result of Bowden, Hagen, King and Reinders. Their result does not hold for arbitrary ordered cycles. Then we give an algorithm to determine the entries of flow-up classes with smallest leading entry on arbitrary ordered cycles. This is an example of Section~\ref{Sflowup}.
In this section we assume that the base ring $R$ is a PID. Before we start to talk about flow-up classes on cycles, we first illustrate ordered and arbitrary ordered cycles in Figure~\ref{sp666}.

\begin{figure}[H]
\begin{center}
\scalebox{0.16}{\includegraphics{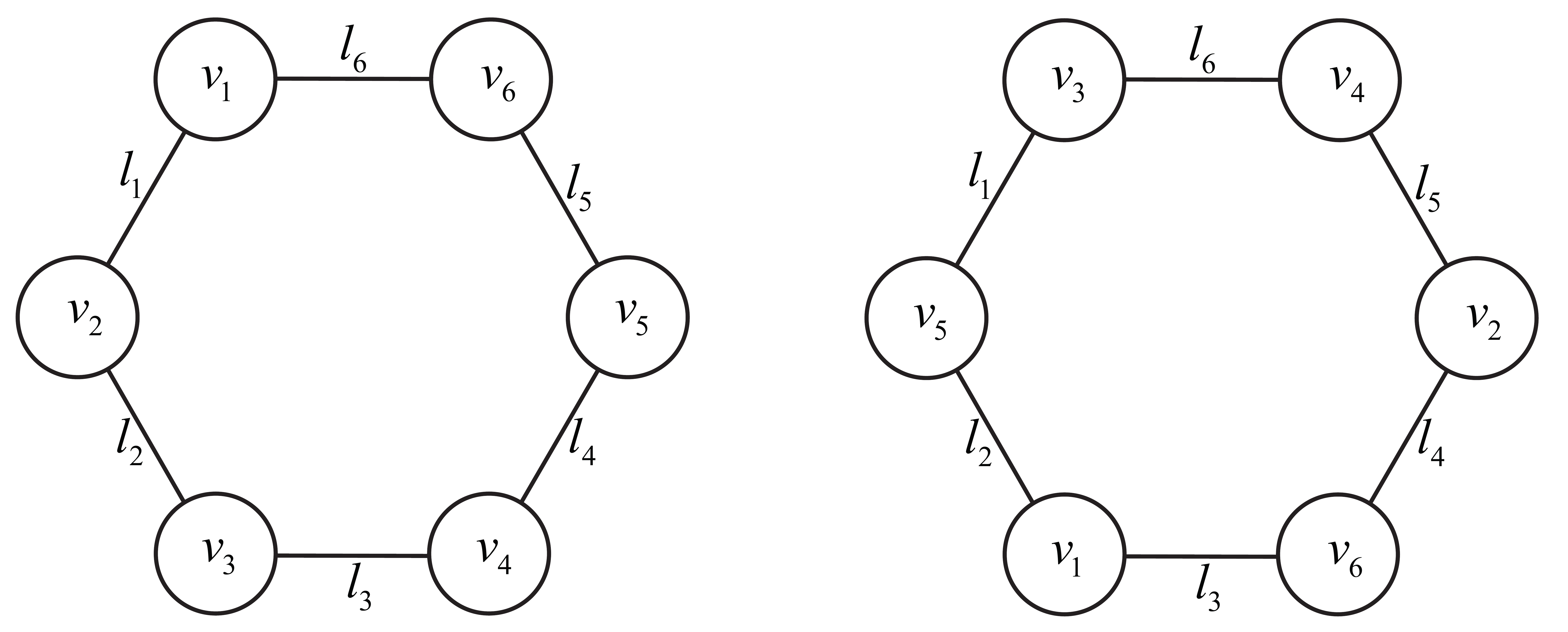}}
\caption{Ordered $6$-cycle (left) and arbitrary ordered cycle (right)}
\label{sp666}
\end{center}
\end{figure}

The following theorem is given for ordered cycles in~\cite{Bow}. \vspace{.3cm}

\begin{theo}~\cite{Bow} Fix an edge-labeled cycle $(C_n , \alpha)$. The vector $F^{(k)} = (0, \ldots , 0 , f_k , \ldots , f_n)$ for $1 < k \leq n$ has entries as follows:
\begin{itemize}
	\item $f_k = \big[ l_{k-1} , (l_k, \ldots ,l_n) \big]$.
	\item For $k < i \leq n$, if $\dfrac{l_{i-1}}{(l_{i-1}, \ldots ,l_n)} = 1$ then $f_i = (l_i , \ldots ,l_n)$.
	\item For $k < i \leq n$, if $\dfrac{l_{i-1}}{(l_{i-1}, \ldots ,l_n)} \neq 1$ then
	\begin{center}
	$f_i = f_{i-1} \cdot \dfrac{(l_i , \ldots ,l_n)}{(l_{i-1}, \ldots ,l_n)} \cdot \bigg( \dfrac{(l_i , \ldots ,l_n)}{(l_{i-1}, \ldots ,l_n)} \bigg)^{-1} _{\emph{ mod } \frac{l_{i-1}}{(l_{i-1}, \ldots ,l_n)}}$
	\end{center}
\end{itemize}
Then $F^{(k)} \in \mathbb{Z}_{(C_n , \alpha)}$.
\label{ordered}
\end{theo}

\begin{proof} See Definition 4.1 and Theorem 4.2 in~\cite{Bow}.
\end{proof}

Bowden, Hagen, King and Reinders showed that the vector $F^{(k)} \in \mathbb{Z}_{(C_n , \alpha)}$ and then they proved that the set $\{ F^{(1)} , \ldots , F^{(n)} \}$ form a basis for the set of splines on $(C_n , \alpha)$ where $F^{(1)} = (1,\ldots ,1)$. They used the fact that on an ordered cycle all entries of a flow-up class can be determined by two modular equations in Equation~\ref{eq3}. If $F^{(i)} = (0 , \ldots , 0 , f_i , \ldots , f_n)$ is an $i$-th flow-up class on an ordered cycle, then for $k = i$ the leading entry $f_i = \big[ l_{i-1} , (l_i, \ldots ,l_n) \big]$ and the other entries $f_k (k> i)$can be determined inductively by the following two modular equations.
\begin{equation}
\begin{aligned}
f_k &\equiv f_{k-1} \text{ mod } l_{k-1} \\
f_k &\equiv 0 \text{ mod } (l_k, \ldots , l_n).
\end{aligned}
\label{eq3}
\end{equation}
We will give a method to determine the entries of a flow-up class with the smallest leading entry on an arbitrary ordered cycle. We first give an algorithm which simplifies the computations.

\begin{alg} Let $C_n$ be an arbitrary cycle and $(C_n , \alpha)$ be an edge labeled graph. Let $F^{(i)} = (0 , \ldots , 0 , f_i , \ldots , f_n)$ be an $i$-th flow-up class with $i \geq 3$. Then
\begin{enumerate}[label=\emph{(\arabic*)}]
	\item Connect the nearest left and the nearest right zero labeled vertices to $v_i$ by an edge. If these two vertices are not adjacent then this operation splits $(C_n , \alpha)$ into two cycles. Otherwise go to (3).
	\item Label all vertices contained in the cycle that not contain $v_i$ by zero.
	\item Contract $(C_n , \alpha)$ by combining all zero labeled vertices and form a new cycle that contain $v_i$ with a single zero labeled vertex.
	\item Fix $f_i = \big[ \big\{ \big( \textbf{p} ^{(i,0)} \big) \big\} \big]$ on the new cycle and determine the other vertex labels by starting the smallest index to the largest one on the new cycle by using Proposition~\ref{suffprop} below and the Chines Remainder Theorem.
\end{enumerate}
\label{alg10}
\end{alg}

As a result of Algorithm~\ref{alg10}, we reduce the problem of finding the entries of a flow-up class on $C_n$ to a cycle with one single zero labeled vertex. The next proposition illustrates that it is sufficient to consider two trails to determine the entries of a flow-up class on an arbitrary ordered cycle. We can also give a generalization of the formula in Theorem~\ref{ordered} as follows.

\begin{prop} Let $(C_n , \alpha)$ be an edge labeled graph with $C_n$ is an arbitrary ordered cycle. Let $F^{(i)}$ be an $i$-th flow-up class  and $v_j$ be a vertex with $j \geq i$. Then it is sufficient to consider the trails $\emph{\textbf{p}} ^{(j,j_l)}$ and $\emph{\textbf{p}} ^{(j,j_r)}$ to determine $f_j$ with $i \leq j \leq n$ inductively, where $v_{j_l}$ is the nearest left vertex to $v_j$ with $j_l < j$ and $v_{j_r}$ is the nearest right vertex to $v_j$ with $j_r < j$.
\label{suffprop}
\end{prop}

\begin{proof} In order to determine $f_j$, we need to consider all $v_k$-trails of $v_j$ with $k \leq j$. Since $C_n$ is a cycle, such trails include either $\textbf{p}^{(j,j_l)}$ or $\textbf{p}^{(j,j_r)}$. Hence it is sufficient to consider $\textbf{p}^{(j,j_l)}$ and $\textbf{p}^{(j,j_r)}$ to determine $f_j$ by Lemma~\ref{sufflem}.
\end{proof}

\begin{theo} Let $(C_n , \alpha)$ be an edge labeled  arbitrary ordered cycle. Let $F^{(i)}$ be an $i$-th flow-up class  and $v_j$ be a vertex with $j \geq i$. The entry $f_j$ of $F^{(i)}$ can be given by the following formula:
\begin{itemize}
\item For $j = i$, $f_j = [(\textbf{p}^{(j,j_r)}) , (\textbf{p}^{(j,j_l)})]$.
\item For $j > i$, $f_j = f_{j_r} + (\textbf{p}^{(j,j_r)}) \cdot \dfrac{f_{j_l} - f_{j_r}}{d} \cdot \left[ \dfrac{(\textbf{p}^{(j,j_r)})}{d} \right]^{-1} _{\emph{ mod } \frac{(\textbf{p}^{(j,j_l)})}{d}}$
\end{itemize}
where $d = \Big( (\textbf{p}^{(j,j_r)}) , (\textbf{p}^{(j,j_l)}) \Big)$.
\label{formula}
\end{theo}

\begin{proof} If $j = i$, we obtain
\begin{displaymath}
\begin{aligned}
f_j &\equiv 0 \text{ mod } (\textbf{p}^{(j,j_r)})\\
f_j &\equiv 0 \text{ mod } (\textbf{p}^{(j,j_l)})
\end{aligned}
\end{displaymath}
and we can set $f_j = [(\textbf{p}^{(j,j_l)}) , (\textbf{p}^{(j,j_r)})]$. 

Assume that $j > i$, then $f_j$ can be given by two modular forms as
\begin{displaymath}
\begin{aligned}
f_j &\equiv f_{j_r} \text{ mod } (\textbf{p}^{(j,j_r)}) \\
f_j &\equiv f_{j_l} \text{ mod } (\textbf{p}^{(j,j_l)}).
\end{aligned}
\end{displaymath}
by Lemma~\ref{sufflem}. We can rewrite this system of congruences as $f_j = s \cdot (\textbf{p}^{(j,j_r)}) + f_{j_r} \equiv f_{j_l} \text{ mod } (\textbf{p}^{(j,j_l)})$. Notice that $f_{j_r}$ and $f_{j_l}$ are already determined since $j_r , j_l < j$. Then we have 
\begin{displaymath}
s \cdot (\textbf{p}^{(j,j_r)}) \equiv f_{j_l} - f_{j_r} \text{ mod } (\textbf{p}^{(j,j_l)}).
\end{displaymath}
Let $d$ be the greatest common divisor of $(\textbf{p}^{(j,j_l)})$ and $(\textbf{p}^{(j,j_r)})$, then 
\begin{displaymath}
s \cdot \dfrac{(\textbf{p}^{(j,j_r)})}{d} \equiv \dfrac{f_{j_l} - f_{j_r}}{d} \text{ mod } \dfrac{(\textbf{p}^{(j,j_l)})}{d}.
\end{displaymath}
By multiplying this equation with the inverse of $(\textbf{p}^{(j,j_r)})$ modulo $(\textbf{p}^{(j,j_l)})$, we obtain
\begin{displaymath}
s \equiv \dfrac{f_{j_l} - f_{j_r}}{d} \left[ \dfrac{(\textbf{p}^{(j,j_r)})}{d} \right]^{-1} \text{ mod } \dfrac{(\textbf{p}^{(j,j_l)})}{d}.
\end{displaymath}
Finally we can get the formula for $f_j$ by setting $s = \left[ \dfrac{(\textbf{p}^{(j,j_r)})}{d} \right]^{-1} _{\text{ mod } \frac{(\textbf{p}^{(j,j_l)})}{d}}$.
\end{proof}

The following example is an application of Algorithm~\ref{alg10} and above observation.

\begin{ex}Consider the arbitrary ordered $8$-cycle in Figure~\ref{sp84}.

\begin{figure}[H]
\begin{center}
\scalebox{0.16}{\includegraphics{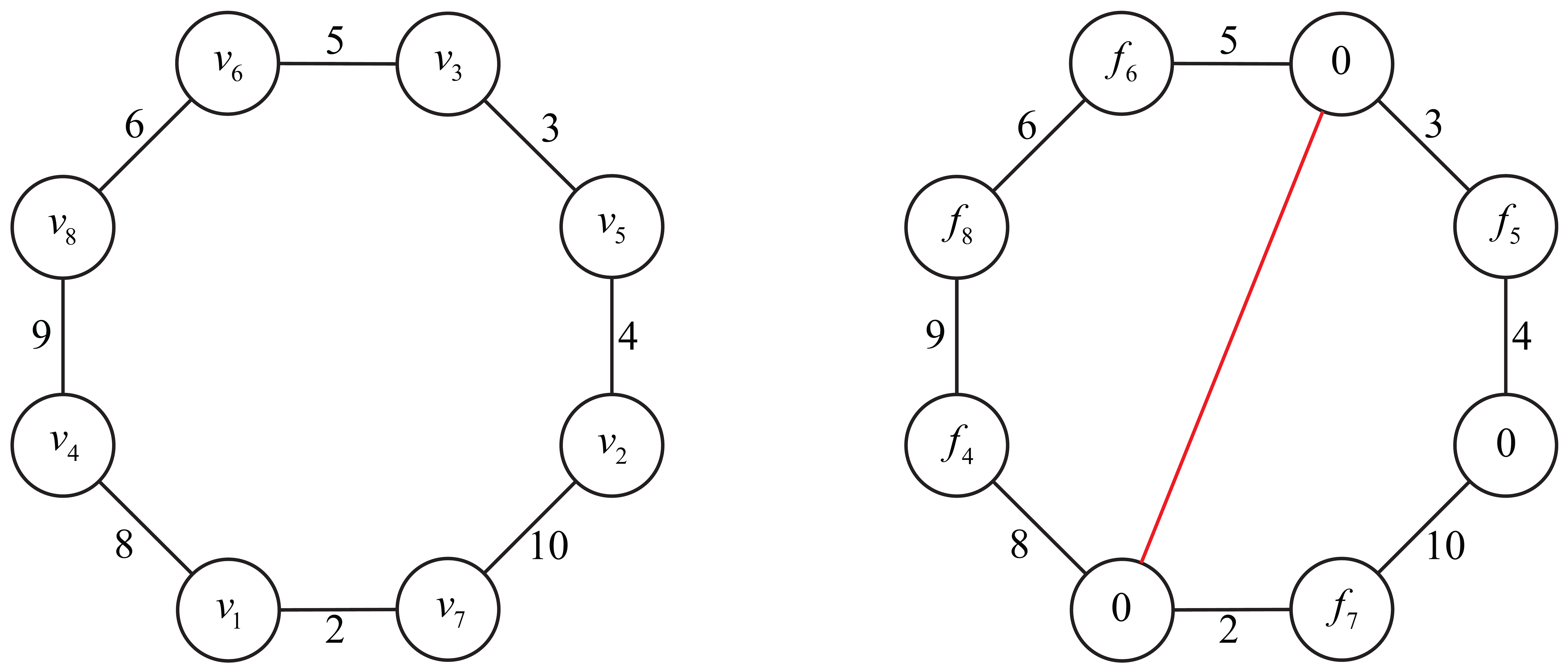}}
\caption{Arbitrary ordered $C_8$ and $F{(4)}$}
\label{sp84}
\end{center}
\end{figure}

\noindent
Let $F^{(4)} \in \mathscr{F}_4$. The nearest left and right zero labeled vertices to $v_4$ are $v_1$ and $v_3$. By Algorithm~\ref{alg10}, we first connect $v_1$ and $v_3$. This gives us two new cycles. We consider the cycle which does not contain $v_4$. Then we label the vertices of this cycle by zero. In this case, $f_5 = f_7 = 0$. We obtain a new cycle containing $v_4$ by combining $v_1$ and $v_3$ as in Figure~\ref{sp83}.
\begin{figure}[H]
\begin{center}
\scalebox{0.16}{\includegraphics{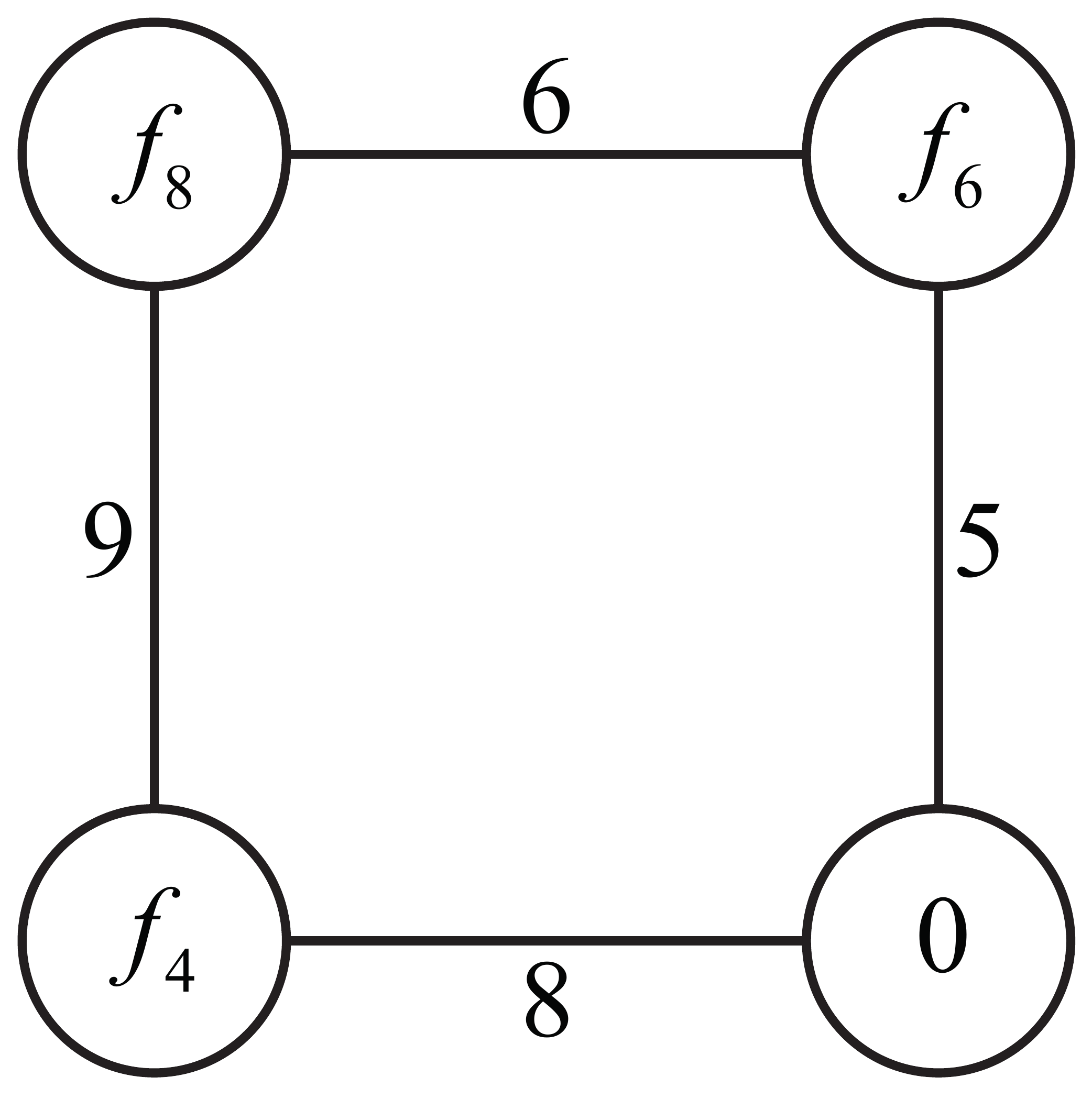}}
\caption{Arbitrary ordered $C_4$}
\label{sp83}
\end{center}
\end{figure}
\noindent
We fix $f_4 = [8 , (9 , 6 , 5)] = 8$ by zero trails. For $f_6$, we have
\begin{displaymath}
\begin{aligned}
f_6 &\equiv 2 \text{ mod } 3 \\
f_6 &\equiv 0 \text{ mod } 5 .
\end{aligned}
\end{displaymath}
The inverse of $3$ modulo $5$ is $2$ and $d = (3,5) = 1$, so we have 
\begin{displaymath}
f_6 = 2 + 3 \cdot (-2) \cdot 2 = -10 \equiv 5 \text{ mod } 15
\end{displaymath}
and we can set $f_6 = 5$. For $f_8$, we have
\begin{displaymath}
\begin{aligned}
f_8 &\equiv 8 \text{ mod } 9 \\
f_8 &\equiv 5 \text{ mod } 6 .
\end{aligned}
\end{displaymath}
Here $d = (9,6) = 3$ and the inverse of $9/3 = 3$ modulo $6/3 = 2$ is $1$, thus
\begin{displaymath}
f_8 = 8 + (-1) \cdot 9 = -1 \equiv 17 \text{ mod } 18
\end{displaymath}
and we can assign $f_8 = 17$. Hence we get $F^{(4)} = (0,0,0,8,0,5,0,17)$. Similarly, we can easily obtain

\begin{displaymath}
\begin{aligned}
F^{(5)}& = (0,0,0,0,12,0,0,0)\\
F^{(6)}& = (0,0,0,0,0,15,0,0)\\
F^{(7)} &= (0,0,0,0,0,0,20,0)\\
F^{(8)}& = (0,0,0,0,0,0,0,18).
\end{aligned}
\end{displaymath}

\end{ex}


\begin{thebibliography}{99} 
\bibitem{Alf} P. Alfeld, \emph{On the dimension of multivariate piecewise polynomials}, Numerical analysis (Dundee,1985), Pitman Res. Notes Math. Ser., vol. 140, Longman Sci. Tech., Harlow, pp. 1-23. MR 873098 (88d:41014), \textbf{1986}

\bibitem{Alt} S. Altinok, S. Sarioglan, \emph{Basis Criteria for Generalized Spline Modules via Determinant}, arXiv:1903.08968, \textbf{2019}.

\bibitem{Bil1} L. Billera, \emph{Homology of smooth splines: generic triangulations and a conjecture of Strang}, Trans. Amer. Math. Soc. 310, no. 1, 325-340. MR 965757 (89k:41010), \textbf{1988}

\bibitem{Bil2} L. Billera, L. Rose, \emph{A dimension series for multivariate splines}, Discrete Comput. Geom. 6, no. 2, 107-128. MR 1083627 (92g:41010), \textbf{1991}

\bibitem{Bil} L. J. Billera, L. L. Rose, \emph{Modules of piecewise polynomials and their freeness}, Math. Zeit. 209, 485-497, \textbf{1992}.

\bibitem{Bow} N. Bowden, S. Hagen, M. King, S. Reinders, \emph{Bases and structure constants of generalized splines with integer coefficients on cycles}, arXiv:1502.00176v1, \textbf{2015}.

\bibitem{Tym} N. Bowden and J. Tymoczko, \emph{Splines mod m}, arXiv:1501.02027v1, \textbf{2015}.

\bibitem{Gil} S. Gilbert, S. Polster, J. Tymoczko, \emph{Generalized splines on arbitrary graphs}, Pacific Journal of Mathematics. Vol. 281, No. 2, 333-364, \textbf{2016}.

\bibitem{Hand} M. Handschy, J. Melnick, S. Reinders, \emph{Integer generalized splines on cycles}, arXiv:1409.1481, \textbf{2014}.

\bibitem{Phi} M. Philbin, L. Swift, A. Tammaro, D. Williams, \emph{Splines over integer quotient rings}, arXiv:1706.00105, \textbf{2017}.

\bibitem{Ros1} L. Rose, \emph{Combinatorial and topological invariants of modules of piecewise polynomials}, Adv. Math. 116, no. 1, 34--45. MR 1361478 (97b:13036) \textbf{1995}

\bibitem{Ros2} L. Rose, \emph{Graphs, syzygies, and multivariate splines}, Discrete Comput. Geom. 32, no. 4, 623-637. MR 2096751 (2005g:41024), \textbf{2004}

\bibitem{Sch} H. Schenck, \emph{Homological methods in the theory of splines}, ProQuest LLC, Ann Arbor, MI, Thesis (Ph.D.)-Cornell University. MR 2695686, \textbf{1997}

\end{thebibliography}
\end{document}